\documentclass{amsart}

\RequirePackage{fix-cm}
\RequirePackage[l2tabu, orthodox]{nag}
\RequirePackage[all, warning]{onlyamsmath}

\usepackage{latexsym}
\usepackage{amsmath,amsfonts,amssymb}
\usepackage{amsthm}
\usepackage{graphicx,color,xcolor}
\usepackage{overpic}
\usepackage{subcaption}
\usepackage{braket}
\usepackage[bookmarks=true,bookmarksnumbered=true]{hyperref}
\usepackage{multirow,bigdelim}
\usepackage{enumerate}
\usepackage{cleveref}
\numberwithin{equation}{section}

\newtheorem{theorem}{Theorem}
\newtheorem{lemma}{Lemma}
\newtheorem{corollary}{Corollary}
\newtheorem{proposition}{Proposition}

\theoremstyle{definition}
\newtheorem{definition}{Definition}

\newtheorem{example}{Example}
\newtheorem{remark}{Remark}

\crefname{section}{Section}{Sections}
\crefname{appendix}{Appendix}{Appendices}
\crefname{theorem}{Theorem}{Theorems}
\crefname{lemma}{Lemma}{Lemmas}
\crefname{corollary}{Corollary}	{Corollaries}			
\crefname{proposition}{Proposition}{Propositions}	
\crefname{claim}{Claim}{Claims}
\crefname{conjecture}{Conjecture}{Conjectures}			
\crefname{definition}{Definition}{Definitions}
\crefname{problem}{Problem}{Problems}
\crefname{example}{Example}{Examples}
\crefname{remark}{Remark}{Remarks}
\crefname{figure}{Figure}{Figures}
\crefname{footnote}{Footnote}{Footnotes}
\crefname{equation}{}{}
\crefname{enumi}{}{}
\crefformat{enumi}{(#2#1#3)}
\Crefformat{enumi}{(#2#1#3)}

\newcommand{\QED}{\hfill \ensuremath{\Box}}
\newcommand{\R}{\mathbb{R}}

\newcommand{\ld}{,\ldots,}
\newcommand{\lld}{,\ldots \ldots,}
\newcommand{\ep}{\varepsilon}
\newcommand{\wt}{\widetilde}

\newcommand{\prn}[1]{\left(#1\right)}
\newcommand{\norm}[1]{\left\|#1\right\|}
\newcommand{\abs}[1]{\left|#1\right|}

\newcommand{\D}{\displaystyle}
\newfont{\bg}{cmr9 scaled\magstep2}
\newcommand{\bigzerol}{\smash{\lower1.0ex\hbox{\bg 0}}}
\DeclareMathOperator{\rank}{rank}
\DeclareMathOperator{\corank}{corank}

\DeclareMathOperator{\ke}{Ker}
\DeclareMathOperator{\codim}{codim}

\begin{document}

\title{Simpliciality of strongly convex problems}

\author{Naoki \textsc{Hamada}}
\address{
Artificial Intelligence Laboratory,
Fujitsu Laboratories Ltd.,
Kawasaki 211-8588, Japan\\
RIKEN AIP-Fujitsu Collaboration Center,
RIKEN,
Tokyo 103-0027, Japan}
\email{hamada-naoki@fujitsu.com}

\author{Shunsuke \textsc{Ichiki}}
\address{
Department of Mathematical and Computing Science,
School of Computing,
Tokyo Institute of Technology,
Tokyo 152-8552,
Japan}
\email{ichiki@c.titech.ac.jp}


\begin{abstract}
\ A multiobjective optimization problem is $C^r$ simplicial if the Pareto set and the Pareto front are $C^r$ diffeomorphic to a simplex and, under the $C^r$ diffeomorphisms, each face of the simplex corresponds to the Pareto set and the Pareto front of a subproblem, where $0\leq r\leq \infty$.
In the paper titled ``Topology of Pareto sets of strongly convex problems,'' it has been shown that a strongly convex $C^r$ problem is $C^{r-1}$ simplicial under a mild assumption on the ranks of the differentials of the mapping for $2\leq r \leq \infty$.
On the other hand, in this paper, we show that a strongly convex $C^1$ problem is $C^0$ simplicial under the same assumption.
Moreover, we establish a specialized transversality theorem on generic linear perturbations of a strongly convex $C^r$ mapping $(r\geq 2)$.
By the transversality theorem, we also give an application of singularity theory to a strongly convex $C^r$ problem for $2\leq r \leq \infty$.
\end{abstract}
\maketitle
\section{Introduction}\label{sec:intro}
In this paper, $m$ and $n$ are positive integers, and we denote the index set $\set{1\ld m}$ by $M$.

We consider the problem of optimizing several functions simultaneously.
More precisely, let $f: X \to \R^m$ be a mapping, where $X$ is a given arbitrary set.
A point $x \in X$ is called a \emph{Pareto optimum} of $f$ if there does not exist another point $y \in X$ such that $f_i(y) \leq f_i(x)$ for all $i \in M$ and $f_j(y) < f_j(x)$ for at least one index $j \in M$.
We denote the set consisting of all Pareto optimums of $f$ by $X^*(f)$, which is called the \emph{Pareto set} of $f$.
The set $f(X^*(f))$ is called the \emph{Pareto front} of $f$.
The problem of determining $X^*(f)$ is called the \emph{problem of minimizing $f$}.

Let $f = (f_1\ld f_m): X \to \R^m$ be a mapping, where $X$ is a given arbitrary set.
For a non-empty subset $I = \set{i_1\ld i_k}$ of $M$ such that $i_1 < \dots < i_k$, set
\begin{align*}
    f_I = (f_{i_1}\ld f_{i_k}).
\end{align*}
The problem of determining $X^*(f_I)$ is called a \emph{subproblem} of the problem of minimizing $f$.
Set
\begin{align*}
    \Delta^{m - 1} &= \Set{(w_1, \dots, w_m) \in \R^m | \sum_{i = 1}^m w_i = 1,\ w_i \geq 0}.
 \end{align*}
We also denote a face of $\Delta^{m - 1}$ for a non-empty subset $I$ of $M$ by
\begin{align*}
    \Delta_I = \set{(w_1, \dots, w_m) \in \Delta^{m - 1} | w_i = 0\ (i \not \in I)}.
\end{align*}

For a $C^r$ manifold $N$ (possibly with corners) and a subset $V$ of $\R^\ell$, a mapping $g:N\to V$ is called a \emph{$C^r$ mapping} (resp., a  \emph{$C^r$ diffeomorphism}) if $g:N\to \R^\ell$ is of class $C^r$ (resp.,  $g:N\to \R^\ell$ is a $C^r$ immersion and $g:N\to V$ is a homeomorphism), where $r\geq 1$.
In this paper, $C^0$ mappings and $C^0$ diffeomorphisms are continuous mappings and homeomorphisms, respectively.

By referring to \cite{Hamada2019}, we give the definition of (weakly) simplicial problems in this paper.
\begin{definition}\label{def:simplicial}
    Let $f = (f_1\ld f_m): X \to \R^m$ be a mapping, where $X$ is a subset of $\R^n$.
    The problem of minimizing $f$ is $C^r$ \emph{simplicial} if there exists a $C^r$ mapping $\Phi: \Delta^{m - 1} \to X^*(f)$ such that both the mappings $\Phi|_{\Delta_I}: \Delta_I \to X^*(f_I)$ and $f|_{X^*(f_I)}: X^*(f_I) \to f(X^*(f_I))$ are $C^r$ diffeomorphisms for any non-empty subset $I$ of $M$, where $0\leq r\leq \infty$.
    The problem of minimizing $f$ is $C^r$ \emph{weakly simplicial}\footnote{In \cite{Hamada2019}, the problem of minimizing $f:X\to \R^m$ is said to be $C^r$ \emph{weakly simplicial} if there exists a $C^r$ mapping $\phi: \Delta^{m - 1} \to f(X^*(f))$ satisfying $\phi(\Delta_I) = f(X^*(f_I))$ for any non-empty subset $I$ of $M$.
    On the other hand, a surjective mapping of $\Delta^{m-1}$ into $X^*(f)$ is important to describe $X^*(f)$. 
Hence, the definition of weak simpliciality in this paper is updated from that in \cite{Hamada2019}.
\label{ftn:weak}
    } if there exists a $C^r$ mapping $\phi: \Delta^{m - 1} \to X^*(f)$ such that $\phi(\Delta_I) = X^*(f_I)$ for any non-empty subset $I$ of $M$,  where $0\leq r\leq \infty$.
\end{definition}
As described in \cite{Hamada2019}, simpliciality is an important property, which can be seen in several practical problems ranging from facility location studied half a century ago~\cite{Kuhn1967} to sparse modeling actively developed today~\cite{Hamada2019}. If a problem is simplicial, then we can efficiently compute a parametric-surface approximation of the entire Pareto set with few sample points~\cite{Kobayashi2019}.

A subset $X$ of $\R^n$ is \emph{convex} if $t x + (1 - t) y \in X$ for all $x, y \in X$ and all $t \in [0, 1]$.
Let $X$ be a convex set in $\R^n$.
A function $f: X \to \R$ is \emph{strongly convex} if there exists $\alpha > 0$ such that
\begin{align*}
    f(t x + (1 - t) y) \leq t f(x) + (1 - t) f(y) - \frac{1}{2} \alpha t (1 - t) \norm{x - y}^2
\end{align*}
for all $x, y \in X$ and all $t \in [0, 1]$, where $\norm{z}$ is the Euclidean norm of $z \in \R^n$.
The constant $\alpha$ is called a \emph{convexity parameter} of the function $f$.
A mapping $f = (f_1\ld f_m): X \to \R^m$ is \emph{strongly convex} if $f_i$ is strongly convex for any $i \in M$.
The problem of minimizing a strongly convex $C^r$ mapping is called the \emph{strongly convex $C^r$ problem}.

In \cite{Hamada2019}, we have the following result for the  simpliciality of strongly convex $C^r$ problems, where $2\leq r\leq \infty$.
\begin{theorem}[\cite{Hamada2019}]\label{thm:main-C2}
    Let $f: \R^n \to \R^m$ be a strongly convex $C^r$ mapping, where $2\leq r\leq \infty$.
    Then, the problem of minimizing $f$ is $C^{r-1}$ simplicial if the rank of the differential $df_x$ is equal to $m - 1$ for any $x \in X^*(f)$.
\end{theorem}
We give the following remark on \cref{thm:main-C2}.
\begin{remark}\label{rem:weak}
    It is shown that if we remove the assumption on the rank of $df_x$ in \cref{thm:main-C2}, then the problem becomes $C^{r-1}$ weakly simplicial in the sense of \cite{Hamada2019} (for the definition of weak simpliciality in the sense of \cite{Hamada2019}, see also \cref{ftn:weak} in this paper).
    In this paper, we show that the problem becomes $C^{r-1}$ weakly simplicial in the sense of \cref{def:simplicial} (for the result, see \cref{thm:weakly_ simplicial} in \cref{sec:weak}).
\end{remark}
As in \cite{Hamada2019}, the assumption $r \geq 2$ is essentially used in the proof of \cref{thm:main-C2}.
It is difficult to apply the same method as in the proof of \cref{thm:main-C2} to strongly convex $C^1$ mappings.
Hence, as the first purpose of this paper, we give a theorem in the case $r=1$ as follows:
\begin{theorem}\label{thm:main-C1}
    Let $f: \R^n \to \R^m$ be a strongly convex $C^1$ mapping.
    Then, the problem of minimizing $f$ is $C^0$ weakly simplicial.
    Moreover, this problem is $C^0$ simplicial if the rank of the differential $df_x$ is equal to $m - 1$ for any $x \in X^*(f)$.
\end{theorem}
In \cite{Hamada2019}, as an application of singularity theory to a strongly convex problem, we have the following result (\cref{thm:H_generic}) on generic linear perturbations of a strongly convex $C^r$ mapping $(2\leq r \leq \infty)$.
Here, note that strong convexity is preserved under linear perturbations (see \cref{thm:preserve_strong} in \cref{sec:app}).
Let $\mathcal{L}(\R^n, \R^m)$ be the space consisting of all linear mappings of $\R^n$ into $\R^m$.
In what follows we will regard $\mathcal{L}(\R^n, \R^m)$ as the Euclidean space $(\R^n)^m$ in the obvious way.

\begin{theorem}[\cite{Hamada2019}]\label{thm:H_generic}
    Let $f: \R^n \to \R^m$ $(n \geq m)$ be a strongly convex $C^r$ mapping, where $2\leq r \leq \infty$.
    If $n - 2m + 4 > 0$, then there exists a Lebesgue measure zero subset $\Sigma $ of $\mathcal{L}(\R^n, \R^m)$ such that for any $\pi \in \mathcal{L}(\R^n, \R^m) - \Sigma$, the problem of minimizing $f + \pi: \R^n \to \R^m$ is $C^{r-1}$ simplicial.
\end{theorem}
In \cref{thm:H_generic}, in order to make a given strongly convex $C^r$ problem simplicial, linear perturbations of all functions $f_1\ld f_m$ are considered, where $f_1\ld f_m$ are the components of $f$.
On the other hand, as the second purpose of this paper, we show that it is sufficient to consider linear perturbations of only $m - 1$ functions (see \cref{thm:maingeneric}).

Let $s$ be an arbitrary integer satisfying $1 \leq s \leq m$.
Set
\begin{align*}
    \mathcal{L}(\R^n, \R^m)_s = \set{(\pi_1\ld \pi_m)\in \mathcal{L}(\R^n, \R^m) | \pi_s = 0}.
\end{align*}
\begin{theorem}\label{thm:maingeneric}
    Let $f: \R^n \to \R^m$ $(n \geq m)$ be a strongly convex $C^r$ mapping, where  $2\leq r \leq \infty$.
    Let $s$ be an arbitrary integer satisfying $1 \leq s \leq m$.
    If $n - 2m + 4 > 0$, then there exists a Lebesgue measure zero subset $\Sigma $ of $\mathcal{L}(\R^n, \R^m)_s$ such that for any $\pi \in \mathcal{L}(\R^n, \R^m)_s - \Sigma$, the problem of minimizing $f + \pi: \R^n \to \R^m$ is $C^{r-1}$ simplicial.
\end{theorem}
In this paper, in order to prove \cref{thm:maingeneric}, we also give a specialized transversality theorem on generic linear perturbations of a strongly convex mapping (see \cref{thm:transverse} in \cref{sec:app}).
Hence, \cref{thm:maingeneric} is also an application of singularity theory to a strongly convex problem.

The remainder of this paper is organized as follows.
In \cref{sec:example}, some examples of (weakly) simplicial problems and remarks on \cref{thm:main-C1,thm:maingeneric} are presented.
By lemmas prepared in \cref{sec:pre}, we prove \cref{thm:main-C1} in \cref{sec:mainproof}.
Moreover, in \cref{sec:app}, preliminaries for the proof of \cref{thm:maingeneric} are given, where the specialized transversality theorem (\cref{thm:transverse}) is shown.
By the transversality theorem, we show \cref{thm:maingeneric} in \cref{sec:maingenericproof}.
\cref{sec:appendix} is an appendix for \cref{rem:weak} and \cref{thm:norm_strong} (for \cref{thm:norm_strong}, see \cref{sec:example}). 
\section{Examples of (weakly) simplicial problems and remarks on \texorpdfstring{\cref{thm:main-C1,thm:maingeneric}}{Theorems 2 and 4}}\label{sec:example}
First, we give some examples of (weakly) simplicial problems.
In order to show given mappings are strongly convex, we prepare \cref{thm:norm_strong}, which is a well-known result.
For the sake of readers' convenience, the proof of \cref{thm:norm_strong} is given in \cref{sec:norm_strong}.

Let $X$ be a convex subset of $\R^n$.
A function $f: X \to \R$ is said to be \emph{convex} if
\begin{align*}
    f(t x + (1 - t) y) \leq t f(x) + (1 - t) f(y)
\end{align*}
for all $x, y \in X$ and all $t \in [0, 1]$.
\begin{lemma}
\label{thm:norm_strong}
    Let $X$ be a convex subset of $\R^n$.
    Then, a function $f: X \to \R$ is strongly convex with a convexity parameter $\alpha > 0$ if and only if the function $g: X \to \R$ defined by $g(x) = f(x) - \frac{\alpha}{2} \norm{x}^2$ is convex.
\end{lemma}

\begin{example}\label{ex:standard1}
    Let $f = (f_1, f_2, f_3): \R^3 \to \R^3$ be the mapping defined by
    \begin{align*}
        f_1(x_1, x_2, x_3) &= a(x_1 - 1)^2 + x_2^2 + x_3^2 \quad (a > 0),\\
        f_2(x_1, x_2, x_3) &= x_1^2 + (x_2 - 1)^2 + x_3^2,\\
        f_3(x_1, x_2, x_3) &= x_1^2 + x_2^2 + (x_3 - 1)^2.
    \end{align*}
    First, we show that $f$ is strongly convex.
    
    Let $\wt{f}:\R^3\to \R$ be the mapping defined by $\wt{f}(x)=\sum_{i=1}^3c_i(x_i-p_i)^2$, where $c_i>0$ for any $i=1, 2, 3$, $x=(x_1,x_2,x_3)$ and $(p_1,p_2,p_3)\in \R^3$.
    Set $\alpha=\min\set{c_1,c_2,c_3}$ and $g(x)= \wt{f}(x)-\frac{\alpha}{2}\norm{x}^2$.
    Then, we have
    \begin{align*}
        g(x)
        =\sum_{i=1}^3\left(\left(c_i-\frac{\alpha}{2}\right)x_i^2-2c_ip_ix_i+c_ip_i^2\right).
    \end{align*}
    Since $c_i-\frac{\alpha}{2}>0$ for all $i=1,2,3$, the function $g$ is convex.
    Therefore, $\wt{f}$ is a strongly convex function with a convexity parameter $\alpha$ by \cref{thm:norm_strong}.
    
    Since $\wt{f}$ is strongly convex,  $f$ is also strongly convex for all $a > 0$.
    Since $\rank df_x \geq 2$ for any $x \in \R^3$ and $a > 0$, the problem of minimizing $f$ is $C^\infty$  simplicial for any $a > 0$ by \cref{thm:main-C2} (see \cref{fig:dsq}).
    With the parameter $a$, the shapes of the Pareto set and the Pareto front change while the simpliciality is maintained.
    If $a = 1$, the Pareto set is a triangle as shown in \cref{fig:dsq1-X}.
    If $a = 4$ or $a = 1/4$, the Pareto set is a curved triangle as shown in \cref{fig:dsq2-X,fig:dsq3-X}.
    For the precise description of $X^*(f)$, see \cref{rem:description} in \cref{sec:mainproof}.
    \begin{figure}[hpbt]
        \centering%
        \subcaptionbox{Simplex $\Delta^2$.\label{fig:simplex}}{%
            \begin{overpic}[width=0.5\hsize,clip,trim=20 20 20 -20]{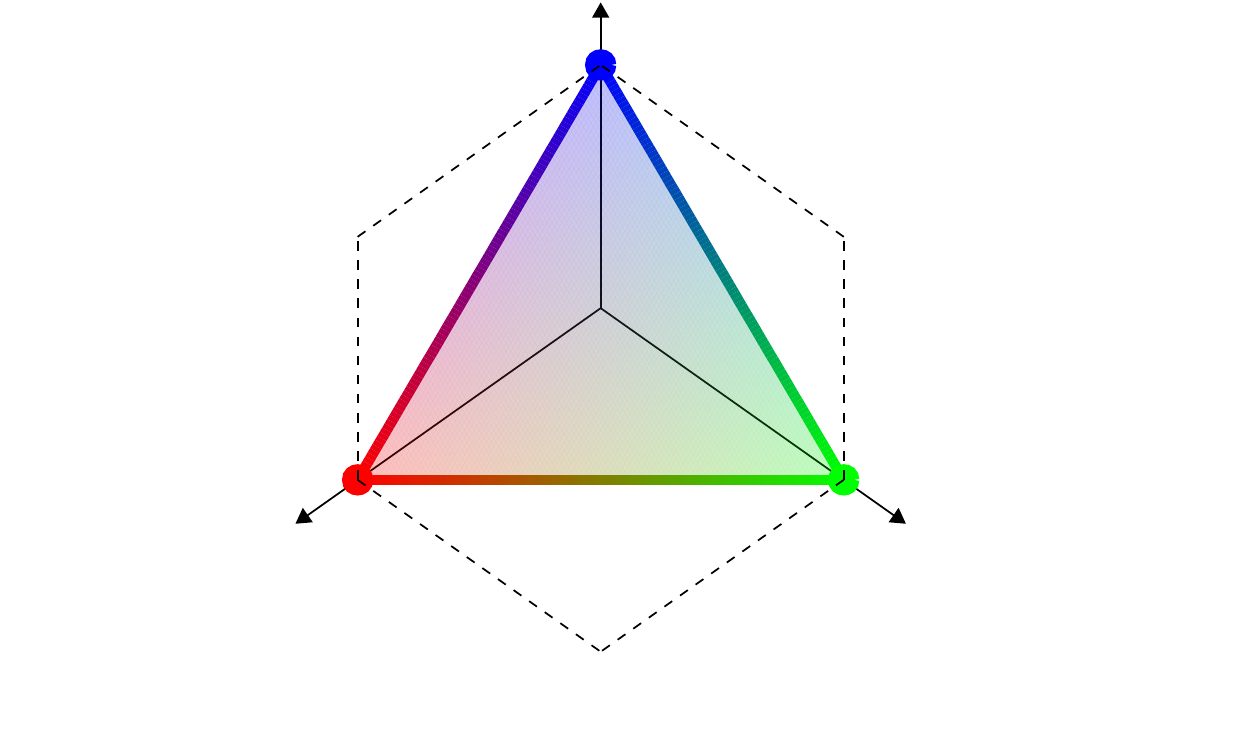}
                \put(47, 29){\tiny $0$}
                \put(25, 13){\tiny $1$}
                \put(69, 13){\tiny $1$}
                \put(50, 56){\tiny $1$}
                \put(14, 11){\tiny $w_1$}
                \put(76, 11){\tiny $w_2$}
                \put(46, 63){\tiny $w_3$}
                \put(14, 20){\tiny $\Delta_{\{1\}}$}
                \put(71, 20){\tiny $\Delta_{\{2\}}$}
                \put(36, 57){\tiny $\Delta_{\{3\}}$}
                \put(42, 14){\tiny $\Delta_{\{1, 2\}}$}
                \put(21, 42){\tiny \colorbox{white}{$\Delta_{\{1, 3\}}$}}
                \put(58, 42){\tiny \colorbox{white}{$\Delta_{\{2, 3\}}$}}
                \put(40, 24){\tiny $\Delta_{\{1, 2, 3\}}$}
            \end{overpic}
        }\\
        \subcaptionbox{Pareto set (left) and Pareto front (right) of $f$ with $a = 1$.\label{fig:dsq1-X}}{%
            \begin{overpic}[width=0.5\hsize,clip,trim=20 20 20 -26]{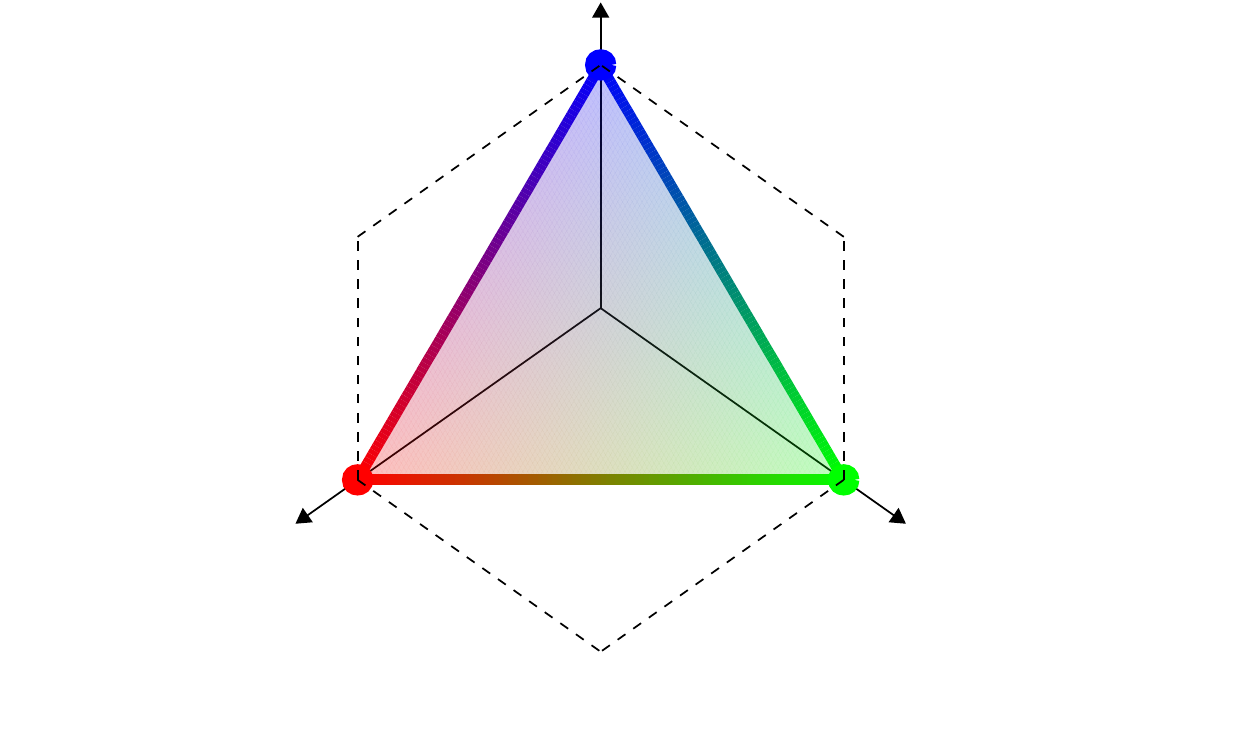}
                \put(47, 29){\tiny $0$}
                \put(25, 13){\tiny $1$}
                \put(69, 13){\tiny $1$}
                \put(50, 56){\tiny $1$}
                \put(14, 11){\tiny $x_1$}
                \put(76, 11){\tiny $x_2$}
                \put(46, 63){\tiny $x_3$}
                \put(05, 20){\tiny $X^*(f_{\{1\}})$}
                \put(71, 20){\tiny $X^*(f_{\{2\}})$}
                \put(28, 57){\tiny $X^*(f_{\{3\}})$}
                \put(37, 14){\tiny $X^*(f_{\{1, 2\}})$}
                \put(13, 42){\tiny \colorbox{white}{$X^*(f_{\{1, 3\}})$}}
                \put(58, 42){\tiny \colorbox{white}{$X^*(f_{\{2, 3\}})$}}
                \put(35, 22){\tiny $X^*(f_{\{1, 2, 3\}})$}
            \end{overpic}
            \begin{overpic}[width=0.5\hsize,clip,trim=20 20 20 -26]{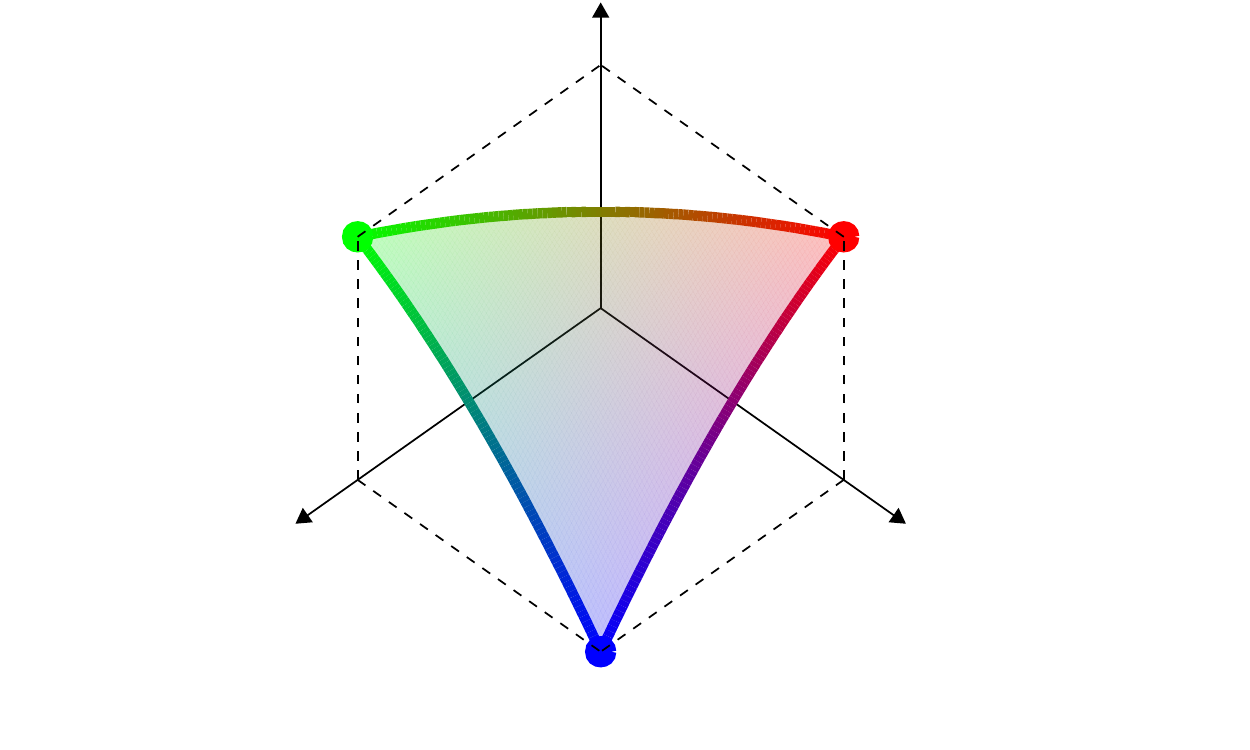}
                \put(47, 29){\tiny $0$}
                \put(25, 13){\tiny $2$}
                \put(69, 13){\tiny $2$}
                \put(50, 56){\tiny $2$}
                \put(14, 11){\tiny $f_1$}
                \put(76, 11){\tiny $f_2$}
                \put(46, 63){\tiny $f_3$}
                \put(70, 44){\tiny $f(X^*(f_{\{1\}}))$}
                \put(02, 44){\tiny $f(X^*(f_{\{2\}}))$}
                \put(51, 02){\tiny $f(X^*(f_{\{3\}}))$}
                \put(34, 46){\tiny \colorbox{white}{$f(X^*(f_{\{1, 2\}}))$}}
                \put(64, 26){\tiny \colorbox{white}{$f(X^*(f_{\{1, 3\}}))$}}
                \put(02, 26){\tiny \colorbox{white}{$f(X^*(f_{\{2, 3\}}))$}}
                \put(32, 36){\tiny $f(X^*(f_{\{1, 2, 3\}}))$}
            \end{overpic}
        }\\
        \subcaptionbox{Pareto set (left) and Pareto front (right) of $f$ with $a = 4$.\label{fig:dsq2-X}}{%
            \begin{overpic}[width=0.5\hsize,clip,trim=20 20 20 -26]{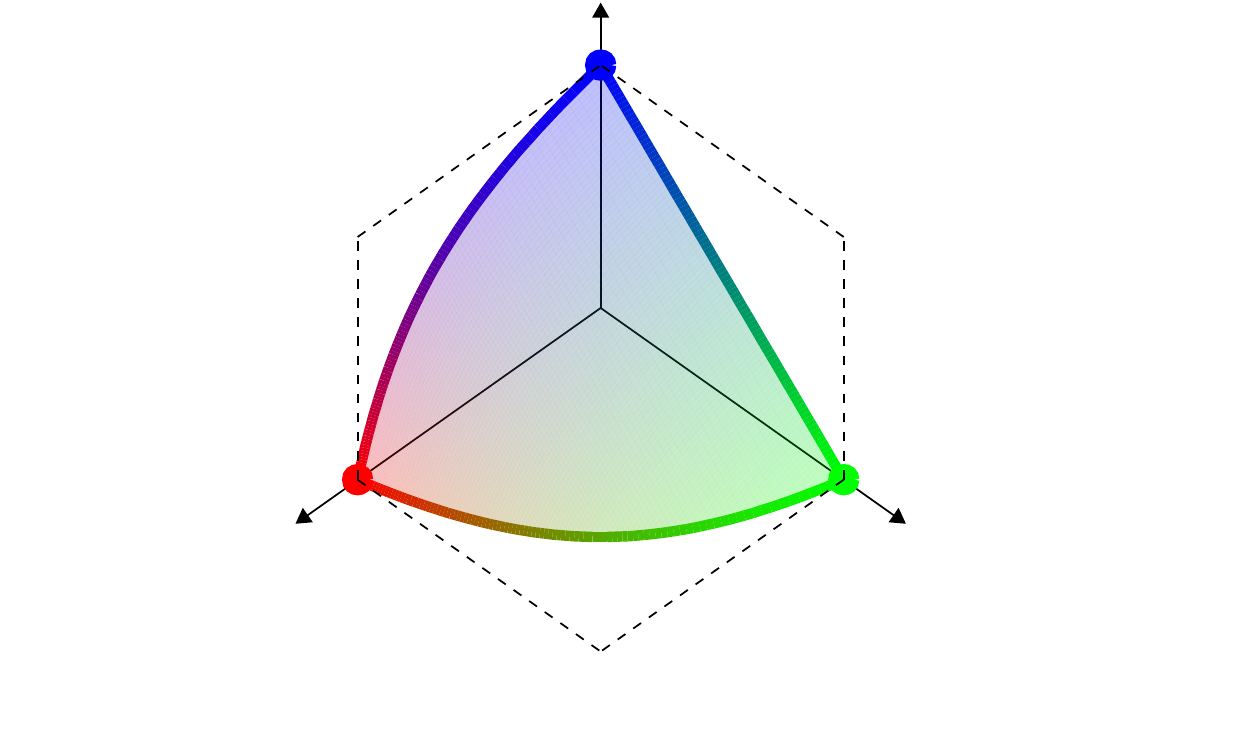}
                \put(47, 29){\tiny $0$}
                \put(25, 13){\tiny $1$}
                \put(69, 13){\tiny $1$}
                \put(50, 56){\tiny $1$}
                \put(14, 11){\tiny $x_1$}
                \put(76, 11){\tiny $x_2$}
                \put(46, 63){\tiny $x_3$}
                \put(06, 20){\tiny $X^*(f_{\{1\}})$}
                \put(72, 18){\tiny $X^*(f_{\{2\}})$}
                \put(28, 57){\tiny $X^*(f_{\{3\}})$}
                \put(37, 08){\tiny \colorbox{white}{$X^*(f_{\{1, 2\}})$}}
                \put(08, 42){\tiny \colorbox{white}{$X^*(f_{\{1, 3\}})$}}
                \put(58, 42){\tiny \colorbox{white}{$X^*(f_{\{2, 3\}})$}}
                \put(36, 21){\tiny $X^*(f_{\{1, 2, 3\}})$}
            \end{overpic}
            \begin{overpic}[width=0.5\hsize,clip,trim=88 58 88 40]{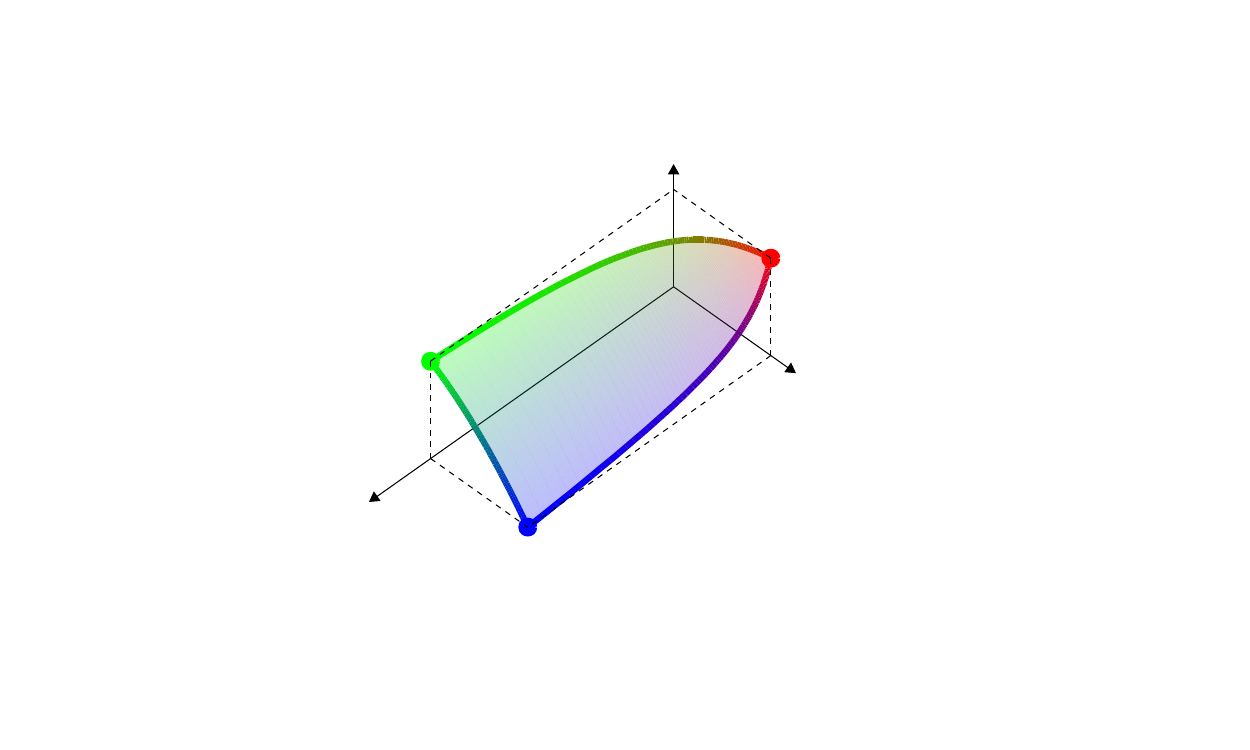}
                \put(56, 36){\tiny $0$}
                \put(19, 09){\tiny $5$}
                \put(72, 25){\tiny $2$}
                \put(59, 56){\tiny $2$}
                \put(05, 04){\tiny $f_1$}
                \put(79, 25){\tiny $f_2$}
                \put(56, 63){\tiny $f_3$}
                \put(72, 49){\tiny $f(X^*(f_{\{1\}}))$}
                \put(00, 34){\tiny $f(X^*(f_{\{2\}}))$}
                \put(38, 02){\tiny $f(X^*(f_{\{3\}}))$}
                \put(17, 49){\tiny $f(X^*(f_{\{1, 2\}}))$}
                \put(54, 14){\tiny $f(X^*(f_{\{1, 3\}}))$}
                \put(-6, 19){\tiny \colorbox{white}{$f(X^*(f_{\{2, 3\}}))$}}
                \put(28, 26){\tiny $f(X^*(f_{\{1, 2, 3\}}))$}
            \end{overpic}
        }\\
        \subcaptionbox{Pareto set (left) and Pareto front (right) of $f$ with $a = 1/4$.\label{fig:dsq3-X}}{%
            \begin{overpic}[width=0.5\hsize,clip,trim=20 20 20 -26]{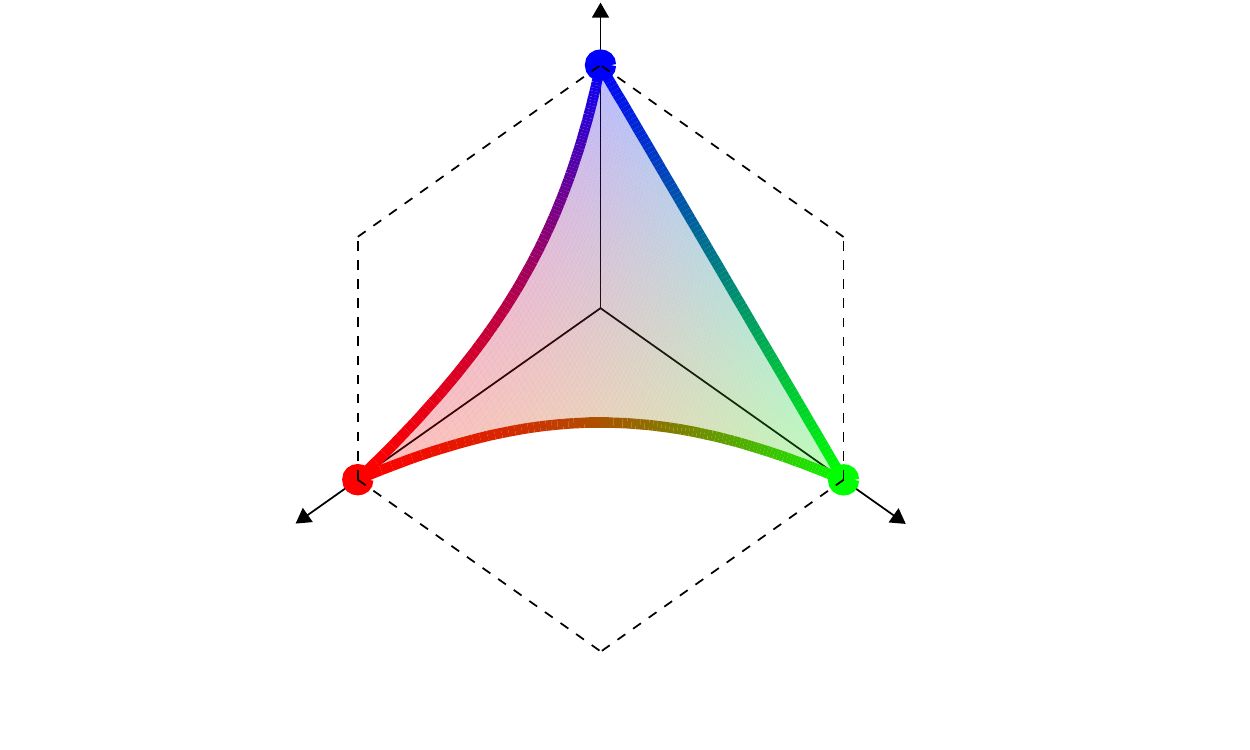}
                \put(47, 29){\tiny $0$}
                \put(25, 13){\tiny $1$}
                \put(69, 13){\tiny $1$}
                \put(50, 56){\tiny $1$}
                \put(14, 11){\tiny $x_1$}
                \put(76, 11){\tiny $x_2$}
                \put(46, 63){\tiny $x_3$}
                \put(06, 20){\tiny $X^*(f_{\{1\}})$}
                \put(72, 18){\tiny $X^*(f_{\{2\}})$}
                \put(28, 57){\tiny $X^*(f_{\{3\}})$}
                \put(38, 18){\tiny $X^*(f_{\{1, 2\}})$}
                \put(17, 42){\tiny \colorbox{white}{$X^*(f_{\{1, 3\}})$}}
                \put(58, 42){\tiny \colorbox{white}{$X^*(f_{\{2, 3\}})$}}
                \put(38, 35){\tiny $X^*(f_{\{1, 2, 3\}})$}
            \end{overpic}
            \begin{overpic}[width=0.5\hsize,clip,trim=34 30 34 -10]{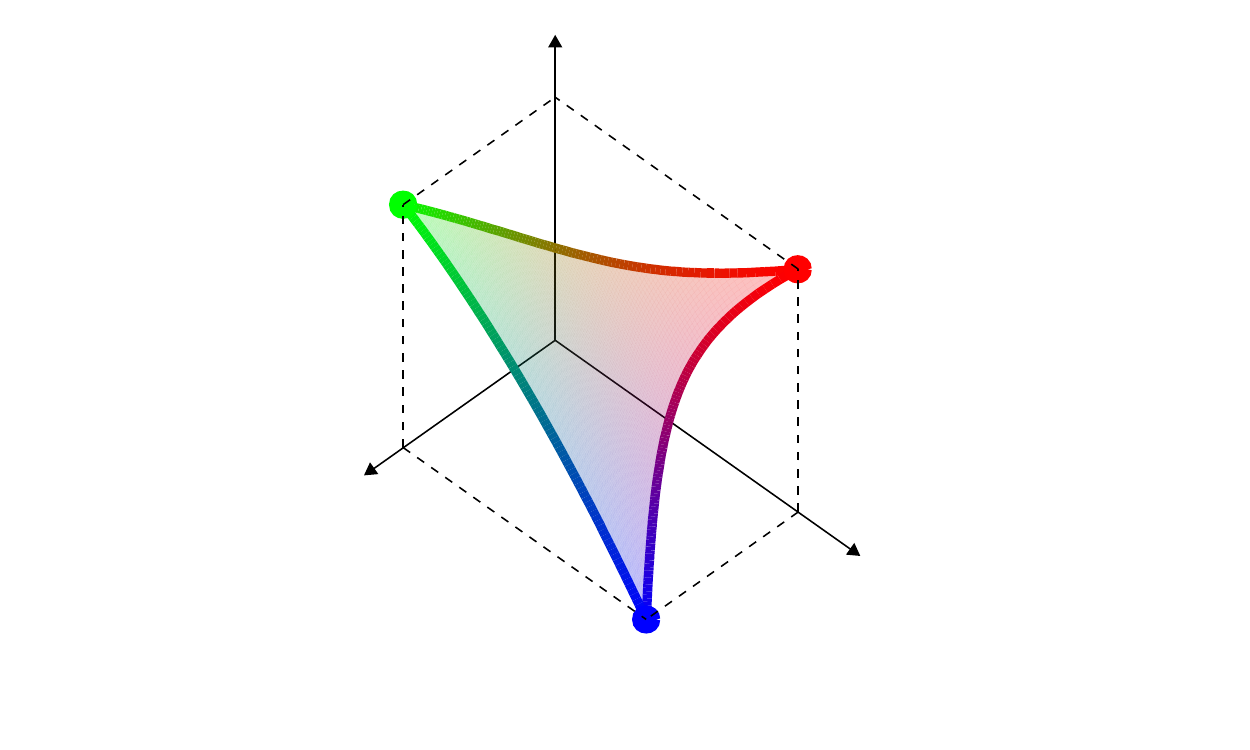}
                \put(42, 26){\tiny $0$}
                \put(27, 14){\tiny $\frac{5}{4}$}
                \put(66, 09){\tiny $2$}
                \put(45, 54){\tiny $2$}
                \put(20, 14){\tiny $f_1$}
                \put(75, 07){\tiny $f_2$}
                \put(41, 63){\tiny $f_3$}
                \put(68, 40){\tiny $f(X^*(f_{\{1\}}))$}
                \put(04, 46){\tiny $f(X^*(f_{\{2\}}))$}
                \put(55, 02){\tiny $f(X^*(f_{\{3\}}))$}
                \put(36, 45){\tiny \colorbox{white}{$f(X^*(f_{\{1, 2\}}))$}}
                \put(57, 22){\tiny \colorbox{white}{$f(X^*(f_{\{1, 3\}}))$}}
                \put(07, 25){\tiny \colorbox{white}{$f(X^*(f_{\{2, 3\}}))$}}
                \put(34, 33){\tiny $f(X^*(f_{\{1, 2, 3\}}))$}
            \end{overpic}
        }%
        \caption{\Cref{ex:standard1} with $a = 1, 4, 1/4$.}\label{fig:dsq}
    \end{figure}
\end{example}
In \cref{ex:notC2}, we give a simple example of a strongly convex $C^1$ mapping which is not of class $C^2$.
\begin{example}\label{ex:notC2}
    Let $f = (f_1, f_2): \R \to \R^2$ be the mapping defined by
    \begin{align*}
        f_1(x) & = (x - 2)^2, \\
        f_2(x) & = \begin{cases}
            x^2             & \text{if $x < 1$},\\
            x^2 + (x - 1)^2 & \text{if $x \geq 1$}.
        \end{cases}
    \end{align*}
    Let $g_i:\R\to\R$ be the function defined by  $g_i(x)=f_i(x)-\frac{2}{2}x^2$, where $i=1,2$.
    Since $g_1$ and $g_2$ are convex, $f_1$ and $f_2$ are strongly convex functions with a convexity parameter $2$ by \cref{thm:norm_strong}, respectively.
    Hence, $f$ is strongly convex.
    Since $f_2$ is not of class $C^2$, we cannot apply \cref{thm:main-C2} to $f$.
    However, since $f$ is of class $C^1$, we can apply \cref{thm:main-C1}.
    Since $\rank df_x = 1$ for any $x \in \R$, the problem of minimizing $f$ is $C^0$ simplicial by \cref{thm:main-C1}.
\end{example}

\begin{remark}\label{rem:main-C1}
We give the following remarks on \cref{thm:main-C1}. 
\begin{enumerate}
  \item 
  Note that (strict) convexity of a mapping does not necessarily imply that the problem is $C^0$ simplicial.
For example, the problem of minimizing $f:\R\to \R$ defined by $f(x)=e^x$ does not have a Pareto optimum (i.e.~a minimizer).
Thus, it is not $C^0$ simplicial although $f$ is strictly convex.
  \item
    We give an example such that \cref{thm:main-C1} does not hold without the rank assumption.
    Let $f = (f_1, f_2): \R \to \R^2$ be the mapping defined by $f(x) = (x^2, x^2)$.
    By \cref{thm:norm_strong}, the mapping $f$ is strongly convex.
    Since $0 \in \R$ is a Pareto optimum and $\rank df_0 = 0$, the mapping $f$ does not satisfy the rank assumption in \cref{thm:main-C1}.
    Since $X^*(f) = \set{0}$, the problem of minimizing $f$ is not $C^0$ simplicial.
    \end{enumerate}
\end{remark}
\begin{remark}\label{rem:maingeneric}
    We give a remark on \cref{thm:maingeneric}. 
    Let $f = (f_1, f_2, f_3): \R^3 \to \R^3$ be the mapping defined by $f_i(x) = \norm{x}^2$ for any integer $i$ $(1 \leq i \leq 3)$.
    By \cref{thm:norm_strong}, the mapping $f$ is strongly convex.
    In order to make the problem of minimizing $f$ simplicial by generic linear perturbations, it is necessary to perturb at least two components of $f$.
    
    First, we consider the case without linear perturbations.
    Since $f_1$, $f_2$ and $f_3$ have the unique minimizer $0 \in \R^3$, we have $X^*(f) = \set{0}$.
    Hence, the problem of minimizing $f$ is not $C^0$ simplicial.
    
    Next, we linearly perturb only one component $f_{s_1}$ of $f$, where $s_1$, $s_2$ and $s_3$ are three elements satisfying $\set{s_1, s_2, s_3} = \set{1, 2, 3}$.
    Set
    \begin{align*}
        \mathcal{L}(\R^3, \R^3)_{(s_2, s_3)}=\set{(\pi_1, \pi_2, \pi_3) \in \mathcal{L}(\R^3, \R^3) | \pi_{s_2} = \pi_{s_3} = 0}.
    \end{align*}
    Let $\pi = (\pi_1, \pi_2, \pi_3)$ be an arbitrary element of $\mathcal{L}(\R^3, \R^3)_{(s_2, s_3)}$.
    Since
    \begin{align*}
        (f_{s_2} + \pi_{s_2})(x) = (f_{s_3} + \pi_{s_3})(x) = \norm{x}^2,
    \end{align*}
    the origin $0 \in \R^3$ is the unique minimizer of $f_{s_2} + \pi_{s_2}$ and $f_{s_3} + \pi_{s_3}$.
    Since $f_{s_1} + \pi_{s_1}$ is a distance-squared function, $f_{s_1} + \pi_{s_1}$ has a unique minimizer.
    Let $p \in \R^3$ be the unique minimizer.
    Then, it is not hard to see that
    \begin{align*}
        X^*(f + \pi) = \set{tp \in \R^3 | t \in [0, 1]}.
    \end{align*}
    Therefore, the problem of minimizing $f + \pi$ is not $C^0$  simplicial.
    
    Finally, we consider linear perturbations of two components of $f$.
    Let $s$ be an arbitrary integer satisfying $1 \leq s \leq 3$.
    By \cref{thm:maingeneric}, there exists a Lebesgue measure zero subset $\Sigma $ of $\mathcal{L}(\R^3, \R^3)_s$ such that for any $\pi \in \mathcal{L}(\R^3, \R^3)_s - \Sigma$, the problem of minimizing $f + \pi: \R^3 \to \R^3$ is $C^\infty$ simplicial.
\end{remark}

\section{Preliminaries for the proof of \texorpdfstring{\cref{thm:main-C1}}{Theorem 2}}\label{sec:pre}
In this section, we prepare some lemmas for the proof of \cref{thm:main-C1}.

Let $f: U \to \R^m$ be a $C^1$ mapping, where $U$ is a non-empty open subset of $\R^n$.
A point $x \in U$ is called a \emph{critical point} of $f$ if $\rank df_x < m$.
We denote the set consisting of all critical points of $f$ by $C(f)$.
The following lemma gives a relationship between critical points and Pareto optimums.
\begin{lemma}\label{thm:critical}
    Let $f: U \to \R^m$ be a $C^1$ mapping, where $U$ is a non-empty open subset of $\R^n$.
    Then, $X^*(f) \subset C(f)$.
\end{lemma}
\begin{proof}[Proof of \cref{thm:critical}]
    In the case $n < m$, since $C(f) = U$, \cref{thm:critical} clearly holds.
    Next, we consider the case $n \geq m$.
    Suppose that there exists $x \in X^*(f)$ such that $x \not \in C(f)$.
    Since $x \not \in C(f)$, there exists an open neighborhood $U_x$ of $x$ such that $f(U_x)$ is an open neighborhood of $f(x)$ by the implicit function theorem.
    This contradicts $x \in X^*(f)$.
\end{proof}

We give the following two lemmas (\cref{thm:sufficient,thm:necessary}) in \cite{Miettinen1999}.
\begin{lemma}[{\cite[Theorem~3.1.3~in~Part~I\hspace{-.1em}I~(p.~79)]{Miettinen1999}}]\label{thm:sufficient}
    Let $f = (f_1, \dots, f_m): \R^n \to \R^m$ be a $($not necessarily continuous$)$  mapping and let $(w_1, \dots, w_m) \in \Delta^{m - 1}$.
    If $x \in \R^n$ is the unique minimizer of the function $\sum_{i = 1}^m w_i f_i$, then $x \in X^*(f)$.
\end{lemma}
The following is a special case of the Karush--Kuhn--Tucker necessary condition for Pareto optimality.

\begin{lemma}[{\cite[Theorem~3.1.5~in~Part~I~(p.~39)]{Miettinen1999}}]\label{thm:necessary}
    Let $f = (f_1, \dots, f_m): \R^n \to \R^m$ be a $C^1$ mapping.
    If $x \in X^*(f)$, then there exists an element $(w_1, \dots, w_m) \in \Delta^{m - 1}$ satisfying $\sum_{i = 1}^m w_i (df_i)_x = 0$.
\end{lemma}

Now, we prepare the following four lemmas (\cref{thm:minimum,thm:chara,thm:weight,thm:lem-injective-strongly}) on strongly convex mappings.

\begin{lemma}[{\cite[Theorem~2.2.6~(p.~85)]{Nesterov2004}}]\label{thm:minimum}
    A strongly convex $C^1$ function $f: \R^n \to \R$ has a unique minimizer.
\end{lemma}

\begin{lemma}[{\cite[Theorem~2.1.9~(p.~64)]{Nesterov2004}}]\label{thm:chara}
    A $C^1$ function $f: \R^n \to \R$ is strongly convex with a convexity parameter $\alpha > 0$ if and only if
    \begin{align*}
        f(x) + df_x \cdot (y - x) + \frac{\alpha}{2} \norm{y - x}^2 \leq f(y)
    \end{align*}
    for any $x, y \in \R^n$.
\end{lemma}
\begin{lemma}[{\cite[Lemma~2.1.4~(p.~64)]{Nesterov2004}}]\label{thm:weight}
    Let $f_i: \R^n \to \R$ be a strongly convex $C^1$  function with a convexity parameter $\alpha_i > 0$, where $i$ is a positive integer $(1 \leq i \leq m)$.
    Then, for any $w = (w_1, \dots, w_m) \in \Delta^{m - 1}$, the function $\sum_{i = 1}^m w_i f_i: \R^n \to \R$ is a strongly convex $C^1$ function with a convexity parameter $\sum_{i = 1}^m w_i \alpha_i$.
\end{lemma}

\begin{lemma}[\cite{Hamada2019c}]\label{thm:lem-injective-strongly}
    Let $f: X \to \R^m$ be a strongly convex $($not necessarily continuous$)$ mapping, where $X$ is a convex subset of $\R^n$.
    Then, $f|_{X^*(f)}: X^*(f) \to \R^m$ is injective.
\end{lemma}

In order to give the last lemma (\cref{thm:ine}) in this section, which is essentially used in the proof of \cref{thm:main-C1}, we prepare the following three lemmas (\cref{thm:weak,thm:closed,thm:compact}). 

Let $f: X \to \R^m$ be a mapping, where $X$ is a given arbitrary set.
A point $x \in X$ is called a \emph{weakly Pareto optimum} of $f$ if there does not exist another point $y \in X$ such that $f_i(y) < f_i(x)$ for all $i \in M$.
Then, by $X^{\mathrm w}(f)$, we denote the set consisting of all weakly Pareto optimums of $f$.

\begin{lemma}[\cite{Hamada2019c}]\label{thm:weak}
    Let $f: \R^n \to \R^m$ be a strongly convex $($not necessarily continuous$)$ mapping.
    Then, we have $X^*(f) = X^{\mathrm w}(f)$.
\end{lemma}

\begin{lemma}\label{thm:closed}
    Let $f: X \to \R^m$ be a continuous mapping, where $X$ is a topological space.
    Then, $X^{\mathrm w}(f)$ is a closed set of $X$.
\end{lemma}

\begin{proof}[Proof of \cref{thm:closed}]
    For the proof, it is sufficient to show that $X - X^{\mathrm w}(f)$ is open.
    Let $x_0 \in X - X^{\mathrm w}(f)$ be an arbitrary element.
    Then, there exists $\wt{x}_0 \in X$ such that $f_i(\wt{x}_0) < f_i(x_0)$ for any $i \in M$, where $f=(f_1\ld f_m)$.
    Set
    \begin{align*}
        O = \set{(y_1\ld y_m) \in \R^m | f_i(x_0) - \ep_i < y_i\ \mbox{for any $i \in M$}},
    \end{align*}
    where
    \begin{align*}
        \ep_i = \D \frac{f_i(x_0) - f_i(\wt{x}_0)}{2}.
    \end{align*}
    Since $f$ is continuous and $O$ is an open neighborhood of $f(x_0)$, the set $f^{-1}(O)$ is an open neighborhood of $x_0$.
    Since $f^{-1}(O) \subset X - X^{\mathrm w}(f)$, the set $X - X^{\mathrm w}(f)$ is open in $X$.
\end{proof}

\begin{lemma}\label{thm:compact}
    Let $f: \R^n \to \R^m$ be a strongly convex $C^1$ mapping.
    Then, $X^*(f)$ is compact.
\end{lemma}

\begin{proof}[Proof of \cref{thm:compact}]
    By \cref{thm:closed,thm:weak}, it follows that $X^*(f)$ is closed.
    Thus, for the proof, it is sufficient to show that $X^*(f)$ is bounded.
    Let $\alpha_i > 0$ be a convexity parameter of $f_i$, where $f = (f_1\ld f_m)$ and $i \in M$.
    By \cref{thm:minimum}, the function $f_i$ has a unique minimizer for any $i \in M$.
    Let $x_i \in \R^n$ be the unique minimizer of $f_i$.
    Set
    \begin{align*}
        \Omega_i = \Set{x \in \R^n | f_i(x_i) + \frac{\alpha_i}{2} \norm{x - x_i}^2 \leq f_i(x_1)}.
    \end{align*}
    Since every $\Omega_i$ is compact, $\Omega = \bigcup_{i = 1}^m \Omega_i$ is also compact.
    Hence, in order to show that $X^*(f)$ is bounded, it is sufficient to show that $X^*(f) \subset \Omega$.
    Suppose that there exists an element $x' \in X^*(f)$ such that $x' \not \in \Omega$.
    Then, it follows that
    \begin{align}\label{eq:com}
        f_i(x_i) + \frac{\alpha_i}{2} \norm{x' - x_i}^2 > f_i(x_1)
    \end{align}
    for any $i \in M$.
    Since $(df_i)_{x_i} = 0$ for any $i \in M$, by \cref{thm:chara}, we have
    \begin{align}\label{eq:com-2}
        f_i(x_i) + \frac{\alpha_i}{2} \norm{x' - x_i}^2 \leq f_i(x').
    \end{align}
    From \cref{eq:com,eq:com-2}, it follows that $f_i(x') > f_i(x_1)$ for any $i \in M$.
    This contradicts $x' \in X^*(f)$.
\end{proof}
Now, we give a mapping from $\Delta^{m - 1}$ into $X^*(f)$, which is introduced in \cite{Hamada2019}.

Let $w = (w_1\ld w_m) \in \Delta^{m - 1}$.
Since $\sum_{i = 1}^m w_i f_i: \R^n \to \R$ is a strongly convex $C^1$ function by \cref{thm:weight}, the function $\sum_{i = 1}^m w_i f_i$ has a unique minimizer by \cref{thm:minimum}.
By \cref{thm:sufficient}, this minimizer is contained in $X^*(f)$.
Hence, we can define a mapping $x^*: \Delta^{m - 1} \to X^*(f)$ as follows:
\begin{align}\label{eq:map}
    x^*(w) = \arg \min_{x \in \R^n} \prn{\sum_{i = 1}^m w_i f_i(x)},
\end{align}
where $\arg \min_{x \in \R^n} \prn{\sum_{i = 1}^m w_i f_i(x)}$ is the minimizer of $\sum_{i = 1}^m w_i f_i$.

\begin{lemma}\label{thm:ine}
    Let $f = (f_1, \dots, f_m): \R^n \to \R^m$ be a strongly convex $C^1$ mapping.
    Let $\alpha_i > 0$ be a convexity parameter of $f_i$ and $K_i$ be the maximal value of $F_i: X^*(f) \times X^*(f) \to \R$ defined by $F_i(x, y) = \abs{f_i(x) - f_i(y)}$ for any $i \in M$.
    Then, for any $w = (w_1\ld w_m), \wt{w} = (\wt{w}_1\ld \wt{w}_m) \in \Delta^{m - 1}$, we have that
    \begin{align*}
        \norm{x^*(w) - x^*(\wt{w})} \leq \sqrt{\D \frac{K_0}{\alpha_0} \sum_{i = 1}^m \abs{w_i - \wt{w}_i}},
    \end{align*}
    where $\alpha_0 = \min \set{\alpha_1\ld \alpha_m}$ and $K_0 = \max \set{K_1\ld K_m}$.
\end{lemma}
\begin{remark}
    In \cref{thm:ine}, the Pareto set $X^*(f)$ is compact by \cref{thm:compact}.
    Hence, for any $i \in M$, the function $F_i$ has the maximal value $K_i$.
\end{remark}

\begin{proof}[Proof of \cref{thm:ine}]
    Let $w, \wt{w} \in \Delta^{m - 1}$ be arbitrary elements.
    By \cref{thm:weight}, the function $\sum_{i = 1}^m w_i f_i: \R^n \to \R$ (resp., $\sum_{i = 1}^m \wt{w}_i f_i: \R^n \to \R$) is a strongly convex function with a convexity parameter $\sum_{i = 1}^m w_i \alpha_i$ (resp., $\sum_{i = 1}^m \wt{w}_i \alpha_i$).
    Since $x^*(w)$ (resp., $x^*(\wt{w})$) is the minimizer of the function $\sum_{i = 1}^m w_i f_i$ (resp., $\sum_{i = 1}^m \wt{w}_i f_i$), we get $d(\sum_{i = 1}^m w_i f_i)_{x^*(w)} = 0$ (resp., $d(\sum_{i = 1}^m \wt{w}_i f_i)_{x^*(\wt{w})} = 0$).
    Thus, by \cref{thm:chara}, we obtain
    {\small
    \begin{align}\label{eq:ine-1}
        \prn{\sum_{i = 1}^m w_i f_i}(x^*(w)) + \frac{\sum_{i = 1}^m w_i \alpha_i}{2}
        \norm{x^*(\wt{w}) - x^*(w)}^2
        &\leq
        \prn{\sum_{i = 1}^m w_i f_i}(x^*(\wt{w})),
        \end{align}
        \begin{align}\label{eq:ine-2}
        \prn{\sum_{i = 1}^m \wt{w}_i f_i}(x^*(\wt{w})) + \frac{\sum_{i = 1}^m \wt{w}_i \alpha_i}{2}
        \norm{x^*(w) - x^*(\wt{w})}^2
        &\leq
        \prn{\sum_{i = 1}^m \wt{w}_i f_i}(x^*(w)).
    \end{align}
    }By \cref{eq:ine-1,eq:ine-2}, we get
    {\small
    \begin{align}\label{eq:ine-3}
        \frac{\sum_{i = 1}^m w_i \alpha_i}{2}
        \norm{x^*(\wt{w}) - x^*(w)}^2
        &\leq
        \sum_{i = 1}^m w_i \prn{f_i(x^*(\wt{w})) - f_i(x^*(w))},
    \end{align}
    \begin{align}\label{eq:ine-4}
        \frac{\sum_{i = 1}^m \wt{w}_i \alpha_i}{2}
        \norm{x^*(\wt{w}) - x^*(w)}^2
        &\leq
        \sum_{i = 1}^m \wt{w}_i \prn{f_i(x^*(w)) - f_i(x^*(\wt{w}))},
    \end{align}
    }respectively.
    By \cref{eq:ine-3,eq:ine-4}, we have
    {\small
    \begin{align*}
        \frac{\sum_{i = 1}^m (w_i + \wt{w}_i) \alpha_i}{2}
        \norm{x^*(\wt{w}) - x^*(w)}^2
        &\leq
        \sum_{i = 1}^m (w_i - \wt{w}_i)
        (
        f_i(x^*(\wt{w})) - f_i(x^*(w))
        ).
    \end{align*}
    }By the inequality above and $\sum_{i = 1}^m(w_i + \wt{w}_i) = 2$, we obtain
    {\small
    \begin{align}\label{eq:ine-5}
        \alpha_0
        \norm{x^*(\wt{w}) - x^*(w)}^2
        &\leq
        \sum_{i = 1}^m (w_i - \wt{w}_i)
        (
        f_i(x^*(\wt{w})) - f_i(x^*(w))
        ).
    \end{align}
    }We also have
    {\small
    \begin{align*}
        \sum_{i = 1}^m \prn{w_i - \wt{w}_i}
        \prn{f_i(x^*(\wt{w})) - f_i(x^*(w))}
        &\leq
        \sum_{i = 1}^m \abs{w_i - \wt{w}_i}
        \abs{f_i(x^*(\wt{w})) - f_i(x^*(w))}
        \\
        &\leq
        \sum_{i = 1}^m \abs{w_i - \wt{w}_i} K_i.
        \\
        &\leq
        K_0 \sum_{i = 1}^m \abs{w_i - \wt{w}_i}.
    \end{align*}
    }By the inequality above and \cref{eq:ine-5}, we obtain
    \begin{align*}
        \alpha_0 \norm{x^*(w) - x^*(\wt{w})}^2
        \leq
        K_0 \sum_{i = 1}^m \abs{w_i - \wt{w}_i}.
    \end{align*}
    Hence, it follows that
    \begin{align*}
        \norm{x^*(w) - x^*(\wt{w})} \leq
        \sqrt{
        \D \frac{K_0}{\alpha_0} \sum_{i = 1}^m \abs{w_i - \wt{w}_i}
        }.
    \end{align*}
\end{proof}

\section{Proof of \texorpdfstring{\cref{thm:main-C1}}{Theorem 2}}\label{sec:mainproof}

First, we give an essential result for the proof of \cref{thm:main-C1} as follows (for the definition of $x^*:\Delta^{m-1}\to X^*(f)$ in \cref{thm:homeo}, see \cref{eq:map}).
\begin{proposition}\label{thm:homeo}
    Let $f:\R^n \to \R^m$ be a strongly convex $C^1$ mapping.
    Then, the following properties hold.
    \begin{enumerate}[$(1)$]
        \item \label{thm:homeo1}
        The mapping $x^*: \Delta^{m - 1} \to X^*(f)$ is surjective and continuous.
        Moreover, if $\rank df_x = m - 1$ for any $x \in X^*(f)$, then $x^*$ is a homeomorphism.
        \item \label{thm:homeo2}
        The mapping $f|_{X^*(f)}: X^*(f) \to \R^m$ is a homeomorphism into the image.
    \end{enumerate}
\end{proposition}

Thus, \cref{thm:main-C1} follows from \cref{thm:homeo} as follows:
Let $I=\set{i_1\ld i_k}$ $(i_1<\cdots <i_k)$ be an arbitrary non-empty subset of $M$ as in \cref{sec:intro}.
Since $f_I: \R^n \to \R^k$ is a strongly convex $C^1$ mapping, $x^*|_{\Delta_I}: \Delta_I \to X^*(f_I)$ is surjective and continuous by \cref{thm:homeo}~\cref{thm:homeo1}.
Hence, the problem of minimizing $f$ is $C^0$ weakly simplicial.
Next, suppose that $\rank df_x = m - 1$ for any $x \in X^*(f)$.
Since
\begin{align*}
    X^*(f_I) = x^*(\Delta_I) \subset x^*(\Delta^{m - 1}) = X^*(f),
\end{align*}
it follows that $\rank (df_I)_x \geq k - 1$ for any $x \in X^*(f_I)$.
By \cref{thm:critical}, it follows that $\rank (df_I)_x = k - 1$ for any $x \in X^*(f_I)$.
Therefore, by \cref{thm:homeo}~\cref{thm:homeo1}, the mapping $x^*|_{\Delta_I}: \Delta_I \to X^*(f_I)$ is a homeomorphism.
Since $X^*(f_I) \subset X^*(f)$, the mapping $f|_{X^*(f_I)}: X^*(f_I) \to \R^m$ is a homeomorphism into the image.
Thus, the problem of minimizing $f$  is $C^0$ simplicial.

By the argument above, in order to complete the proof of \cref{thm:main-C1}, it is sufficient to show \cref{thm:homeo}.
\begin{proof}[Proof of \cref{thm:homeo}~\cref{thm:homeo1}]
    Note that the bijectivity of $x^*$ is shown by the same method as in the proof of \cite{Hamada2019}.
    For the sake of readers' convenience, we give the proof in this paper.
    
    First, we show that $x^*$ is surjective.
    Let $x \in X^*(f)$ be an arbitrary point.
    By \cref{thm:necessary}, there exists $w = (w_1\ld w_m) \in \Delta^{m - 1}$ such that $\sum_{i = 1}^m w_i (df_i)_x = 0$.
    Namely, we get $d(\sum_{i = 1}^m w_i f_i)_x = 0$.
    Since the function $\sum_{i = 1}^m w_i f_i$ is strongly convex, the point $x$ is the unique minimizer of $\sum_{i = 1}^m w_i f_i$ by \cref{thm:chara}.
    This implies $x^*(w) = x$.
    Hence, $x^*$ is surjective.
    
    Second, we show that $x^*$ is continuous.
    Let $\wt{w} = (\wt{w}_1\ld \wt{w}_m) \in \Delta^{m - 1}$ be an arbitrary element.
    For the proof, it is sufficient to show that $x^*$ is continuous at $\wt{w}$.
    Let $\ep$ be an arbitrary positive real number.
    Then, there exists an open neighborhood $V$ of $\wt{w}$ in $\Delta^{m - 1}$ satisfying
    \begin{align*}
        \sqrt{\D \frac{K_0}{\alpha_0} \sum_{i = 1}^m \abs{w_i - \wt{w}_i}} < \ep
    \end{align*}
    for any $w \in V$, where $K_0$ and $\alpha_0$ are defined in \cref{thm:ine}.
    From \cref{thm:ine}, it follows that
    \begin{align*}
        \norm{x^*(w) - x^*(\wt{w})} < \ep
    \end{align*}
    for any $w \in V$.
    
    Finally, we show that $x^*$ is a homeomorphism if $\rank df_x = m - 1$ for any $x \in X^*(f)$.
    Since $x^*$ is surjective and continuous from a compact space $\Delta^{m - 1}$ into a Hausdorff space, for this proof, it is sufficient to show that $x^*$ is injective.
    
    Suppose that $x^*(w) = x^*(\wt{w})$, where $w = (w_1\ld w_m)$ and $\wt{w} = (\wt{w}_1\ld \wt{w}_m)$.
    Since $x^*(w) \in X^*(f)$ is the unique minimizer of $\sum_{i = 1}^m w_i f_i$, we have \[d \prn{\sum_{i = 1}^m w_i f_i}_{x^*(w)} = 0.\]
    Namely, we get
    \begin{align*}
        (w_1\ld w_m) df_{x^*(w)} = (0\ld 0).
    \end{align*}
    By the above argument, we also have $(\wt{w}_1\ld \wt{w}_m) df_{x^*(\wt{w})} = (0\ld 0)$.
    Since $x^*(w) = x^*(\wt{w})$, we obtain
    \begin{align*}
        (\wt{w}_1\ld \wt{w}_m) df_{x^*(w)} = (0\ld 0).
    \end{align*}
    Since $m = \dim \ke df_{x^*(w)} + \rank df_{x^*(w)}$ and $\rank df_{x^*(w)} = m - 1$, it follows that $\dim \ke df_{x^*(w)} = 1$.
    Since $w, \wt{w} \in \ke df_{x^*(w)} \cap \Delta^{m - 1}$, we obtain $w = \wt{w}$.
\end{proof}
\begin{proof}[Proof of \cref{thm:homeo}~\cref{thm:homeo2}]
    By \cref{thm:homeo}~\cref{thm:homeo1}, $X^*(f)$ $(= x^*(\Delta^{m - 1}))$ is compact.
    By \cref{thm:lem-injective-strongly}, $f|_{X^*(f)}: X^*(f) \to \R^m$ is injective.
    Since $f|_{X^*(f)}: X^*(f) \to f(X^*(f))$ is a bijective and continuous mapping from a compact space into a Hausdorff space, the mapping $f|_{X^*(f)}$ is a homeomorphism onto the image.
\end{proof}

Finally, as supplements to this section, we give the following two remarks.
\begin{remark}
    In \cref{thm:homeo}~\cref{thm:homeo1}, the assumption that $\rank df_x = m - 1$ for any $x \in X^*(f)$ yields $m - 1 \leq n$.
    On the other hand, when $m - 1 > n$, it is impossible that $x^*: \Delta^{m - 1} \to X^*(f) (\subset \R^n)$ is a homeomorphism by the invariance of domain theorem.
    For the invariance of domain theorem, see \cite{Hatcher2002}.
\end{remark}

\begin{remark}
    The mapping $x^*$ in \cref{thm:homeo}~\cref{thm:homeo1} is not necessarily differentiable as follows.
    Let $f = (f_1, f_2): \R \to \R^2$ be the mapping defined in \cref{ex:notC2} of \cref{sec:example}.
    Let $\varphi: [0, 1] \to \Delta^1$ be the diffeomorphism defined by $\varphi(w_1) = (w_1, 1 - w_1)$.
    Since if $x^*(w_1, w_2) = x$ then $d(w_1 f_1 + w_2 f_2)_x = 0$, we can easily obtain the following:
    \begin{align*}
        x^* \circ \varphi(w_1) & =
        \begin{cases}
            2w_1 & \text{if $0 \leq w_1 < \frac{1}{2}$},\\
            \D \frac{w_1 + 1}{-w_1 + 2} & \text{if $\frac{1}{2} \leq w_1 \leq 1$}.\\
        \end{cases}
    \end{align*}
    Since
    \begin{align*}
        \lim_{h \to +0}
        \D \frac{(x^* \circ \varphi)
        \prn{\frac{1}{2} + h}
        -(x^* \circ \varphi)
        \prn{\frac{1}{2}}}{h}
        & = \D \frac{4}{3},\\
        \lim_{h \to -0} \D \frac{(x^* \circ \varphi)
        \prn{\frac{1}{2} + h} - \prn{x^* \circ \varphi}
        \prn{\frac{1}{2}}}{h}& = 2,
    \end{align*}
    the mapping $x^* \circ \varphi$ is not differentiable at $w_1 = \frac{1}{2}$.
\end{remark}
\begin{remark}\label{rem:description}
    The mapping $x^*$ in \cref{thm:homeo}~\cref{thm:homeo1} is useful for describing a Pareto set as follows.
    
    Let $f:\R^3\to \R^3$ be the mapping defined by \cref{ex:standard1}.
    Let $w=(w_1,w_2,w_3)\in \Delta^2$.
    Since $x^*(w)$ is a minimizer of $\sum_{i=1}^3w_if_i$ by the definition of $x^*$, we have $d(\sum_{i=1}^3w_if_i)_{x^*(w)}=0$.
    Thus, by simple calculations, $x^*:\Delta^2\to X^*(f)$ can be described as follows:
    \begin{align*}
        x^*(w_1,w_2,w_3)=\left(\frac{aw_1}{aw_1+(1-w_1)}, w_2, w_3\right).
    \end{align*}
    Since $x^*(\Delta^2)=X^*(f)$, the Pareto set $X^*(f)$ can be described as follows: 
    \begin{align*}
        X^*(f)=\Set{\left(\frac{aw_1}{aw_1+(1-w_1)}, w_2, w_3\right)\in \R^3|(w_1,w_2,w_3)\in \Delta^2}.
    \end{align*}
\end{remark}

\section{Preliminaries for the proof of \texorpdfstring{\cref{thm:maingeneric}}{Theorem 4}}\label{sec:app}
In this section, unless otherwise stated, all manifolds are without boundary and assumed to have countable bases.

The purpose of this section is to establish the specialized transversality theorem (\cref{thm:transverse}) for generically linearly perturbed strongly convex mappings, which is an essential tool for the proof of \cref{thm:maingeneric}.
First, we prepare the following two lemmas.
\begin{lemma}[{\cite[Theorem~2.1.11 (p.~65)]{Nesterov2004}}]\label{thm:hess_strong}
    Let $U$ be a convex open subset of $\R^n$ $(U \neq \emptyset)$.
    A $C^2$ function $f: U \to \R$ is strongly convex with a convexity parameter $\alpha > 0$ if and only if $m(f)_x \geq \alpha$ for any $x \in U$, where $m(f)_x$ is the minimal eigenvalue of the Hessian matrix of $f$ at $x$.
\end{lemma}

\begin{lemma}[\cite{Hamada2019}]\label{thm:preserve_strong}
    Let $f: \R^n \to \R^m$ be a strongly convex mapping.
    Then, for any $\pi \in \mathcal{L}(\R^n, \R^m)$, the mapping $f + \pi: \R^n \to \R^m$ is also strongly convex.
\end{lemma}
For the statement and the proof of \cref{thm:transverse}, we prepare some definitions.
Let $U$ be a non-empty open set of $\R^n$ and $J^1(U, \R^m)$ be the space of $1$-jets of mappings of $U$ into $\R^m$.
Then, note that $J^1(U, \R^m)$ is a $C^\infty$ manifold.
For a given $C^r$ mapping $f: U \to \R^m$ $(r \geq 2)$, the mapping $j^1 f: U \to J^1(U, \R^m)$ is defined by $x \mapsto j^1 f(x)$.
Then, notice that $j^1 f: U \to J^1(U, \R^m)$ is of class $C^{r - 1}$.
Further, set
\begin{align*}
    \Sigma^k = \set{j^1 f(0) \in J^1(n, m) | \corank Jf(0) = k},
\end{align*}
where $J^1(n,m)=\set{j^1f(0)|f:(\R^n,0)\to (\R^m,0)}$, $\corank Jf(0) = \min \set{n, m} - \rank Jf(0)$ and $k = 1 \ld \min \set{n, m}$.
Set
\begin{align*}
    \Sigma^k(U, \R^m) = U \times \R^m \times \Sigma^k.
\end{align*}
Then, the set $\Sigma^k(U, \R^m)$ is a submanifold of $J^1(U, \R^m)$ satisfying
\begin{align*}
    \codim \Sigma^k(U, \R^m)
        &= \dim J^1(U, \R^m) - \dim \Sigma^k(U, \R^m)\\
        &= (n - v + k)(m - v + k),
\end{align*}
where $v = \min \set{n, m}$.
For details on $j^1 f: U \to J^1(U, \R^m)$, $\Sigma^k$ and $\Sigma^k(U, \R^m)$, see \cite{Golubitsky1974}.

Now, we recall the definition of transversality.
\begin{definition}\label{def:transverse}
    \textrm{
    Let $X$ and $Y$ be $C^r$ manifolds, and $Z$ be a $C^r$ submanifold of $Y$ ($r \geq 1$).
    Let $f: X \to Y$ be a $C^1$ mapping.
    \begin{enumerate}
        \item
        We say that $f: X \to Y$ is \emph{transverse} to $Z$ \emph{at $x\in X$} if $f(x) \not \in Z$ or in the case $f(x) \in Z$, the following holds:
        \begin{align*}
            df_x(T_xX) + T_{f(x)}Z = T_{f(x)}Y.
        \end{align*}
        \item
        We say that $f: X \to Y$ is \emph{transverse} to $Z$ if for any $x \in X$, the mapping $f$ is transverse to $Z$ at $x$.
    \end{enumerate}
    }
\end{definition}
The following is the basic transversality result, which is a key lemma for the proof of \cref{thm:transverse}.
\begin{lemma}[\cite{Golubitsky1974,Ichiki2019}]\label{thm:basic}
    Let $X$, $A$ and $Y$ be $ C^r$ manifolds, $Z$ be a $C^r$ submanifold of $Y$ and $\Gamma: X \times A \to Y$ be a $C^r$ mapping.
    If $r > \max \set{\dim X - \codim Z, 0}$ and $\Gamma$ is transverse to $Z$, then there exists a Lebesgue measure zero subset $\Sigma$ of $A$ such that for any $a \in A - \Sigma$, the $C^r$ mapping $\Gamma_a: X \to Y$ is transverse to $Z$, where $\codim Z = \dim Y - \dim Z$ and $\Gamma_a(x) = \Gamma(x, a)$, $x\in X$.
\end{lemma}
In \cite{Golubitsky1974}, \cref{thm:basic} is shown in the case that all manifolds and mappings are of class $C^\infty$.
By the same method, \cref{thm:basic} can be shown (cf.~\cite{Ichiki2019}).

\begin{proposition}\label{thm:transverse}
    Let $f: U \to \R^m$ be a strongly convex $C^r$ mapping, where $U$ is a convex open subset of $\R^n$ $(U \neq \emptyset)$.
    Let $s$ be an arbitrary integer satisfying $1 \leq s \leq m$, and $k$ be an arbitrary integer satisfying $1 \leq k \leq \min \set{n, m}$.
    If
    \begin{align*}
        r > \max \set{n - \codim \Sigma^k(U, \R^m), 0} + 1,
    \end{align*}
    then there exists a Lebesgue measure zero subset $\Sigma$ of $\mathcal{L}(\R^n, \R^m)_s$ such that for any $\pi \in \mathcal{L}(\R^n, \R^m)_s - \Sigma$, the mapping $j^1(f + \pi): U \to J^1(U, \R^m)$ is transverse to $\Sigma^k(U, \R^m)$.
\end{proposition}

\begin{remark}
    We give an example such that \cref{thm:transverse} does not hold without the hypothesis of strong convexity.
    Let $f = (f_1, f_2): \R^2 \to \R^2$ be the mapping defined by $f_1(x_1, x_2) = 0$ and $f_2(x_1, x_2) = x_1^2 + x_2^2$.
    Note that $f_1$ is not strongly convex by \cref{thm:norm_strong}.
    Let $\pi = (\pi_1, \pi_2) \in \mathcal{L}(\R^2, \R^2)_1$ be an arbitrary element.
    Then, it follows that $j^1(f + \pi)(p) \in \Sigma^2(\R^2, \R^2)$ and $\rank d(j^1(f + \pi))_p \leq 2$, where $p$ is the unique minimizer of $f_2 + \pi_2$.
    Since $\codim \Sigma^2(\R^2, \R^2) = 4$, the mapping $j^1(f + \pi)$ is not transverse to $\Sigma^2(\R^2, \R^2)$.
\end{remark}

\begin{proof}[Proof of \cref{thm:transverse}]
    In the case $m=1$, \cref{thm:transverse} clearly holds by \cref{thm:hess_strong}.
    
    Hence, we will consider the case $m\geq 2$.
    For a positive integer $\ell$, we denote the $\ell \times \ell$ unit matrix by $E_\ell$.
    For simplicity, set
    \begin{align*}
    A = \mathcal{L}(\R^n, \R^m)_s.
    \end{align*}
    In order to show \cref{thm:transverse}, it is sufficient to give the proof in the case $s = 1$.
    
    Let $\Gamma: U \times A \to J^1(U, \R^m)$ be the $C^{r - 1}$ mapping defined by
    \begin{align*}
        \Gamma(x, \pi) = j^1(f + \pi)(x).
    \end{align*}
    Note that $r - 1 > \max \set{n - \codim \Sigma^k(U, \R^m), 0}$.
    If $\Gamma$ is transverse to $\Sigma^k(U, \R^m)$, then there exists a Lebesgue measure zero subset $\Sigma$ of $A$ such that for any $\pi \in A-\Sigma$, the mapping $\Gamma_\pi: U \to J^1(U, \R^m)$ is transverse to $\Sigma^k(U, \R^m)$ by \cref{thm:basic}, where $\Gamma_\pi(x) = \Gamma(x, \pi)$, $x\in U$.
    Thus, in order to finish the proof, it is sufficient to show that $\Gamma$ is transverse to $\Sigma^k(U, \R^m)$.
    Let $(\wt{x}, \wt{\pi}) \in U \times A$ be an arbitrary element satisfying $\Gamma(\wt{x}, \wt{\pi}) \in \Sigma^k(U, \R^m)$.
    Then, it is sufficient to show that
    \begin{align}\label{eq:transverse1}
        \dim \prn{d \Gamma_{\prn{\wt{x}, \wt{\pi}}} \prn{T_{\prn{\wt{x}, \wt{\pi}}}(U \times A)} + T_{\Gamma(\wt{x}, \wt{\pi})}\Sigma^k(U, \R^m)}
        = n + m + nm.
    \end{align}
    
    Let $(a_{ij})_{1 \leq i \leq m, 1 \leq j \leq n}$ be a representing matrix of a linear mapping $\pi \in A$.
    Since $s = 1$, note that $a_{1j} = 0$ for any $j$ $(1 \leq j \leq n)$.
    Thus, $f + \pi: U \to \R^m$ is given as follows:
    \begin{align*}
        (f + \pi)(x) = \prn{
            f_1(x), f_2(x) + \sum_{j = 1}^n a_{2j} x_j
            \ld
            f_m(x) + \sum_{j = 1}^n a_{m j} x_j
        },
    \end{align*}
    where $f = (f_1\ld f_m)$, $x = (x_1\ld x_n)$ and $(a_{21}\ld a_{2n}\ld a_{m 1}\ld a_{m n}) \in (\R^n)^{m - 1}$.
    
    Hence, the mapping $\Gamma$ is given by
    \begin{align*}
    &\Gamma(x, \pi)
    \\
    &=\left(x, (f + \pi)(x), \frac{\partial f_1}{\partial x_1}(x)\ld \frac{\partial f_1}{\partial x_n}(x), \right.
    \\
    &\qquad \left. \frac{\partial f_2}{\partial x_1}(x) + a_{21}
    \ld
    \frac{\partial f_2}{\partial x_n}(x) + a_{2n}
    \lld
    \frac{\partial f_m}{\partial x_1}(x) + a_{m1}
    \ld
    \frac{\partial f_m}{\partial x_n}(x) + a_{mn}
    \right).
    \end{align*}
    The Jacobian matrix of $\Gamma$ at $\prn{\wt{x}, \wt{\pi}}$ is as follows:
    \begin{align*}
    J \Gamma_{(\wt{x}, \wt{\pi})} =
    \left(
    \begin{array}{@{\,}c@{\,\,}|@{\,\,}c@{\,\,\,}c@{\,\,\,}c@{\,\,\,}c@{\,\,\,}c}
    E_n       &        & &            \\
     \ast         &        &         \bigzerol        &      \\
    H(f_1)_{\wt{x}}       &           &          & \\
    \\[-3.3mm] \hline \\[-3.3mm]
     \ast    &  E_n             &    & \bigzerol        \\
     \vdots   &   \bigzerol  & \ddots &      \\
       \ast     &  &  &   E_n \\
    \end{array}
    \right),
    \end{align*}
    where $H(f_1)_{\wt{x}}$ is the Hessian matrix of $f_1$ at $\wt{x}$.
    Notice that there are $m - 1$ copies of $E_n$ in the lower right partition of the above description of $J \Gamma_{(\wt{x}, \wt{\pi})}$.
    Since $\Sigma^k(U, \R^m)$ is a sub-bundle of $J^1(U, \R^m)$ with the fiber $\Sigma^k$, in order to show \cref{eq:transverse1}, it is sufficient to show that the matrix $R$ has rank $n + m + nm$:
    \begin{align*}
    R=
    \left(
    \begin{array}{@{\,}c@{\,\,}|@{\,\,}c@{\,\,}|@{\,\,\,}c@{\,\,\,}c@{\,\,\,}c@{\,\,\,}c}
    E_{n + m}     &    \ast    &             &   0       &   \\
    \hline    0  &   H(f_1)_{\wt{x}}  & & 0 &  \\
     \hline     
      &      \ast        & E_n      &       &   \bigzerol \\
   \bigzerol     &  \vdots  & \bigzerol &\ddots &      \\
          & \ast   &  &  &   E_n \\
    \end{array}
    \right).
    \end{align*}
    Notice that there are $m - 1$ copies of $E_n$ in the above description of $R$.
    Note that for any $i$ $(1 \leq i \leq nm)$, the $(n + m + i)$-th column vector of $R$ coincides with the $i$-th column vector of $J \Gamma_{(\wt{x}, \wt{\pi})}$.
    Since $f_1$ is a strongly convex $C^2$ function, we have $\rank H(f_1)_{\wt{x}} = n$ by \cref{thm:hess_strong}.
    Hence, it follows that $\rank R = n + m + nm$.
    Therefore, we obtain \cref{eq:transverse1}.
\end{proof}

\section{Proof of \texorpdfstring{\cref{thm:maingeneric}}{Theorem 4}}\label{sec:maingenericproof}
Since \cref{thm:maingeneric} clearly holds by combining the following result (\cref{thm:transversecoro}) and \cref{thm:main-C2}, in order to show \cref{thm:maingeneric}, it is sufficient to prove \cref{thm:transversecoro}.
\begin{corollary}\label{thm:transversecoro}
    Let $f: \R^n \to \R^m$ $(n \geq m)$ be a strongly convex $C^r$ mapping $(r \geq 2)$.
    Let $s$ be an arbitrary integer satisfying $1 \leq s \leq m$.
    If $n - 2m + 4 > 0$, then there exists a Lebesgue measure zero subset $\Sigma$ of $\mathcal{L}(\R^n, \R^m)_s$ such that for any $\pi \in \mathcal{L}(\R^n, \R^m)_s - \Sigma$ and any $x \in \R^n$, we have $\rank d(f + \pi)_x \geq m - 1$.
\end{corollary}
\begin{proof}[Proof of \cref{thm:transversecoro}]
    In the case $m = 1$, \cref{thm:transversecoro} clearly holds.
    
    Hence, we consider the case $m \geq 2$.
    Since $n \geq m$, we have
    \begin{align*}
        \codim \Sigma^2(\R^n, \R^m) = 2(n - m + 2).
    \end{align*}
    Since $n - 2m + 4 > 0$, we also have $\codim \Sigma^2(\R^n, \R^m) > n$.
    
    Let $k$ be an arbitrary integer satisfying $2 \leq k \leq m$.
    It follows that
    \begin{align}\label{eq:codim}
        n - \codim \Sigma^k(\R^n, \R^m) \leq n - \codim \Sigma^2(\R^n, \R^m) < 0.
    \end{align}
    Furthermore, we have
    \begin{align*}
        r \geq 2 > \max \set{n - \codim \Sigma^k(\R^n, \R^m), 0} + 1.
    \end{align*}
    By \cref{thm:transverse}, there exists a Lebesgue measure zero subset $\Sigma_k$ of $\mathcal{L}(\R^n, \R^m)_s$ such that for any $\pi \in \mathcal{L}(\R^n, \R^m)_s - \Sigma_k$, the mapping $j^1(f + \pi)$ is transverse to $\Sigma^k(\R^n, \R^m)$.
    Set $\Sigma = \bigcup_{k = 2}^m \Sigma_k $.
    Then, $\Sigma$ has Lebesgue measure zero in $\mathcal{L}(\R^n, \R^m)_s$.
    
    Let $\pi \in \mathcal{L}(\R^n, \R^m)_s - \Sigma$ and $x \in \R^n$ be arbitrary elements.
    Suppose $\rank d(f + \pi)_x \leq m - 2$.
    Then, there exists an integer $k$ $(2 \leq k \leq m)$ satisfying
        $j^1(f + \pi)(x) \in \Sigma^k(\R^n, \R^m)$.
    Since the mapping $j^1(f + \pi)$ is transverse to $\Sigma^k(\R^n, \R^m)$, we obtain
    \begin{align*}
        d(j^1(f + \pi))_x(T_{x} \R^n) + T_{j^1(f + \pi)(x)} \Sigma^k(\R^n, \R^m) = T_{j^1(f + \pi)(x)} J^1(\R^n, \R^m).
    \end{align*}
    This equation implies that
    \begin{align*}
        \dim d(j^1(f + \pi))_x(T_{x} \R^n) \geq \codim \Sigma^k(\R^n, \R^m).
    \end{align*}
    This contradicts \cref{eq:codim}.
\end{proof}

\section{Appendix}\label{sec:appendix}
\subsection{On \texorpdfstring{\cref{rem:weak}}{Remark 1}}\label{sec:weak}

As described in \cref{rem:weak}, we show that the problem of minimizing a strongly convex $C^r$ mapping $f:\R^n\to\R^m$ ($2\leq r\leq \infty$) becomes $C^{r-1}$ weakly simplicial in the sense of \cref{def:simplicial} as follows. 
\begin{theorem}\label{thm:weakly_ simplicial}
 Let $f: \R^n \to \R^m$ be a strongly convex $C^r$ mapping, where $2\leq r\leq \infty$.
    Then, the problem of minimizing $f$ is $C^{r-1}$ weakly simplicial.
\end{theorem}
In order to show \cref{thm:weakly_ simplicial}, we prepare the following result in \cite{Hamada2019}.
\begin{proposition}[\cite{Hamada2019}]\label{thm:Pareto_convex}
Let $f = (f_1, \dots, f_m): \R^n \to \R^m$ be a strongly convex $C^r$ mapping $(2 \le r \le \infty)$.
Then, $x^*: \Delta^{m - 1} \to X^*(f)$ is a surjective mapping of class $C^{r - 1}$.
\end{proposition}
\begin{proof}[Proof of \cref{thm:weakly_ simplicial}]
Let $I=\set{i_1\ld i_k}$ $(i_1<\cdots <i_k)$ be an arbitrary non-empty subset of $M$ as in \cref{sec:intro}.
Since $f_I: \R^n \to \R^k$ is a strongly convex $C^r$ mapping, $x^*|_{\Delta_I}: \Delta_I \to X^*(f_I)$ is a surjective mapping of class $C^{r-1}$ by \cref{thm:Pareto_convex}, where $2\leq r\leq \infty$.
Hence, the problem of minimizing $f$ is $C^{r-1}$ weakly simplicial.
\end{proof}
\subsection{Proof of \texorpdfstring{\cref{thm:norm_strong}}{Lemma 1}}\label{sec:norm_strong}In order to show \cref{thm:norm_strong}, we prepare the following lemma.
\begin{lemma}\label{thm:norm_strong2}
    For any $t\in\R$ and any $x, y\in \R^n$, we have 
    \begin{align*}
        t\norm{x}^2+(1-t)\norm{y}^2-\norm{tx+(1-t)y}^2=t(1-t)\norm{x-y}^2.
    \end{align*}
\end{lemma}
\begin{proof}[Proof of \cref{thm:norm_strong2}]
We have
{\small
    \begin{align*}
        t\norm{x}^2+(1-t)\norm{y}^2-\norm{tx+(1-t)y}^2&=t\sum_{i=1}^nx_i^2+(1-t)\sum_{i=1}^ny_i^2-\sum_{i=1}^n(tx_i+(1-t)y_i)^2
        \\
        &=t(1-t)\sum_{i=1}^n(x_i^2+y_i^2-2x_iy_i)
        \\
        &=t(1-t)\norm{x-y}^2,
    \end{align*}
    }where $x=(x_1\ld x_n)$ and $y=(y_1\ld y_n)$.
\end{proof}

Now, we will prove \cref{thm:norm_strong}.
A mapping $f: X \to \R$ is strongly convex with a convexity parameter $\alpha > 0$ if and only if for all $t\in [0,1]$ and all $x, y\in X$, we have 
    \begin{align}\label{eq:ap1}
    f(t x + (1 - t) y) &\leq t f(x) + (1 - t) f(y) - \frac{1}{2} \alpha t (1 - t) \norm{x - y}^2.
    \end{align}
By \cref{thm:norm_strong2}, the inequality \cref{eq:ap1} holds for all $t\in [0,1]$ and all $x, y\in X$ if and only if we have 
{\small
    \begin{align}\label{eq:ap2}
    f(t x + (1 - t) y) 
   &\leq t f(x) + (1 - t) f(y) - \frac{1}{2} \alpha\left( t\norm{x}^2+(1-t)\norm{y}^2-\norm{tx+(1-t)y}^2\right),
    \end{align}
    }for all $t\in [0,1]$ and all $x, y\in X$.
 The inequality \cref{eq:ap2} holds for all $t\in [0,1]$ and all $x, y\in X$ if and only if we have 
 {\small
\begin{align}\label{eq:ap3}
    f(t x + (1 - t) y)-\frac{1}{2} \alpha\norm{tx+(1-t)y}^2 \leq t \left(f(x)-\frac{1}{2} \alpha\norm{x}^2\right) +(1-t)\left(f(y)-\frac{1}{2} \alpha\norm{y}^2\right)
    \end{align}
    }for all $t\in [0,1]$ and all $x, y\in X$. 
The inequality \cref{eq:ap3} holds for all $t\in [0,1]$ and all $x, y\in X$ if and only if the function $g: X \to \R$ defined by $g(x) = f(x) - \frac{\alpha}{2} \norm{x}^2$ is convex.\QED

\section*{Acknowledgements}
The authors are most grateful to the anonymous reviewer for his/her careful
reading of the first manuscript of this paper and invaluable suggestions.
They are grateful to \mbox{Kenta~Hayano}, \mbox{Yutaro~Kabata} and \mbox{Hiroshi~Teramoto} for their kind comments.
\mbox{Shunsuke~Ichiki} was supported by JSPS KAKENHI Grant Numbers JP19J00650 and JP17H06128.
This work is based on the discussions at 2018 IMI Joint Use Research Program, Short-term Joint Research ``Multiobjective optimization and singularity theory: Classification of Pareto point singularities'' in Kyushu University.
This work was also supported by the Research Institute for Mathematical Sciences, a Joint Usage/Research Center located in Kyoto University.

\bibliographystyle{plain}
\bibliography{main}
\end{document}